\documentclass[10pt]{amsart}
\usepackage{mathrsfs}
\usepackage{amscd}
\usepackage{enumerate}
\usepackage[bb=fourier,cal=pxtx]{mathalfa}
\usepackage{amssymb}
\usepackage{mathtools}      
\usepackage{nccmath}        
\usepackage{bm}             
\usepackage{indentfirst}    
\usepackage{amsthm}
\usepackage{thmtools}
\usepackage{enumitem}       
\usepackage[backref=page,colorlinks=true]{hyperref} 
\usepackage[font=small]{caption}
\usepackage{tikz}
\usepackage{tikz-cd}        

\renewcommand*{\backref}[1]{}
\renewcommand*{\backrefalt}[4]{\quad \tiny
  \ifcase #1 (\textbf{NOT CITED.})%
  \or    (Cited on page~#2.)%
  \else   (Cited on pages~#2.)%
  \fi}

\makeatletter

\def\MRbibitem{\@ifnextchar[\my@lbibitem\my@bibitem}

\def\mybiblabel#1#2{\@biblabel{{\hyperref{http://www.ams.org/mathscinet-getitem?mr=#1}{}{}{#2}}}}

\def\myhyperanchor#1{\Hy@raisedlink{\hyper@anchorstart{cite.#1}\hyper@anchorend}}

\def\my@lbibitem[#1]#2#3#4\par{%
  \item[\mybiblabel{#2}{#1}\myhyperanchor{#3}\hfill]#4%
  \@ifundefined{ifbackrefparscan}{}{\BR@backref{#3}}%
  \if@filesw{\let\protect\noexpand\immediate
    \write\@auxout{\string\bibcite{#3}{#1}}}\fi\ignorespaces%
}

\def\my@bibitem#1#2#3\par{%
  \refstepcounter\@listctr
  \item[\mybiblabel{#1}{\the\value\@listctr}\myhyperanchor{#2}\hfill]#3%
  \@ifundefined{ifbackrefparscan}{}{\BR@backref{#2}}%
  \if@filesw\immediate\write\@auxout
    {\string\bibcite{#2}{\the\value\@listctr}}\fi\ignorespaces%
}

\makeatother

\hypersetup{bookmarksdepth = 3} 
\numberwithin{equation}{section}     

\setlist[enumerate,1]{label={\upshape(\alph*)},ref=\alph*}
\setlist[enumerate,2]{label={\upshape(\arabic*)},ref=\arabic*}

\newcommand{\R}{\mathbb{R}}

\newcommand{\cL}{\mathcal{L}}
\newcommand{\cM}{\mathcal{M}}

\newcommand{\cU}{\mathcal{U}}

\newcommand{\sA}{\mathscr{A}}

\newcommand{\sD}{\mathscr{D}}

\newcommand{\hK}{\widehat{K}}

\DeclareBoldMathCommand\ba{a}
\DeclareBoldMathCommand\bb{b}

\newcommand*\circled[1]{\tikz[baseline=(char.base)]{
    \node[shape=circle,draw,inner sep=1pt] (char) {\footnotesize{#1}};}}

\renewcommand{\epsilon}{\varepsilon}

\renewcommand{\emptyset}{\varnothing}

\newtheorem{prop}{Proposition}[section]

\newtheorem{lemm}[prop]{Lemma}

\theoremstyle{definition}
\newtheorem{defi}[prop]{Definition}

\theoremstyle{remark}

\makeatletter
\DeclareFontFamily{OMX}{MnSymbolE}{}
\DeclareSymbolFont{MnLargeSymbols}{OMX}{MnSymbolE}{m}{n}
\SetSymbolFont{MnLargeSymbols}{bold}{OMX}{MnSymbolE}{b}{n}
\DeclareFontShape{OMX}{MnSymbolE}{m}{n}{
    <-6>  MnSymbolE5
   <6-7>  MnSymbolE6
   <7-8>  MnSymbolE7
   <8-9>  MnSymbolE8
   <9-10> MnSymbolE9
  <10-12> MnSymbolE10
  <12->   MnSymbolE12
}{}
\DeclareFontShape{OMX}{MnSymbolE}{b}{n}{
    <-6>  MnSymbolE-Bold5
   <6-7>  MnSymbolE-Bold6
   <7-8>  MnSymbolE-Bold7
   <8-9>  MnSymbolE-Bold8
   <9-10> MnSymbolE-Bold9
  <10-12> MnSymbolE-Bold10
  <12->   MnSymbolE-Bold12
}{}

\let\llangle\@undefined
\let\rrangle\@undefined
\DeclareMathDelimiter{\llangle}{\mathopen}%
                     {MnLargeSymbols}{'164}{MnLargeSymbols}{'164}
\DeclareMathDelimiter{\rrangle}{\mathclose}%
                     {MnLargeSymbols}{'171}{MnLargeSymbols}{'171}
\makeatother

\begin{document}

\title[Tree pressure]{Tree pressure for hyperbolic and non-exceptional upper semi-continuous potentials}
\date{\today}

\author{Yiwei Zhang}

\begin{thanks}
{Y.Z.\ was supported by project Fondecyt 3130622.}
\end{thanks}



\begin{abstract}
In this note, we investigate the tree pressure for multi-modal
interval maps with a certain class of hyperbolic and non-exceptional
upper semi-continuous potential functions. When restricting to H\"{o}lder potentials, our result recovers
the Corollary 2.2 in \cite{LRL14} for maps with non-flat critical points. On the other hand, this property will also be used to prove the
existence of a conformal measure for the geometric potential in the
negative spectrum.
\end{abstract}

\maketitle



Let us recall some basic setting of thermodynamic formalism,
referring the interested reader to \cite{Ke98} or \cite{PU11} for
more information.

Let $(X,d)$ be a compact metric space, and $T:X\to X$ be a
continuous map with $h_{top}(T)<\infty$. Denote by $\cM(X)$ the
space of Borel probability measures on $X$ endowed with the weak*
topology, and let $\cM(T,X)$ denote the subset of $T-$invariant
ones. For each measure $\nu\in\cM(T,X)$, denote by $h_{\nu}(T)$ the
\emph{measure-theoretic} entropy for $\nu$.

Given a upper semi-continuous\footnote{A function
$\phi:X\to\R\cup\{-\infty\}$ is \emph{upper semi-continuous} if the
sets $\{y\in X: \phi(y)<c\}$ are open for each $c\in\R$. Since $X$
is compact, $\sup\phi<+\infty$.} function
$\phi:X\to\R\cup\{-\infty\}$, the \emph{topological pressure} of $T$
for the \emph{potential} $\phi$ is defined as
$$
P(T,\phi):=\sup\left\{h_{\nu}(T)+\int_{X}\phi
d\nu:~~\nu\in\cM(T,X)\right\}.
$$
An \emph{equilibrium state} of $T$ for the potential $\phi$ is a
measure which attains the supremum.

\bigskip

Let $I$ be a compact interval in $\R$. For a differentiable map
$f:I\to I$, a point of $I$ is \emph{critical} if the derivative of
$f$ vanishes at it. We denote by $\mbox{Crit}(f)$ the set of
critical points of $f$. We also denote by $J(f)$ the \emph{Julia
set}, which is the complement of the largest open subset of $I$ on
which the family of iterates of $f$ is normal. In particular, let
$\mbox{Crit}'(f):=\mbox{Crit}(f)\cap J(f)$.

In what follows, we denote by $\mathscr{A}$ the collection of all
non-injective differentiable maps $f:I\to I$ such that
\begin{itemize}
  \item The critical set is finite;
  \item $Df$ is H\"{o}lder continuous;
  \item For each critical point $c\in\mbox{Crit}(f)$, there exist any number
  $\ell_{c}\geq1$ and diffeomorphisms $\varphi$ and $\psi$ with $D\varphi,D\psi$ H\"{o}lder continuous,
  such that $\varphi(c)=\psi(f(c))=0$, and such that on a neighborhood
  of $c$ on I, we have
  $$
|\psi\circ f|=\pm|\varphi|^{\ell_c};
  $$
  \item The Julia set $J(f)$ is \emph{completely invariant} \footnote{In contrast with complex rational maps, the Julia set of an interval map might not completely
  invariant.
  However, it is possible to make an arbitrarily small smooth perturbation of $f$ outside a neighborhood, so that the Julia set of the perturbed map is completely invariant, and coincides with $J(f)$ correspondingly.} (i.e., $f(J)=f^{-1}(J)=J$), and contains
  at least two points;
  \item All periodic points are \emph{hyperbolic repelling} (i.e., a periodic point $p$ of periodicity $N$ with
$|D(f^{N})(p)|>1$);
  \item $f$ is \emph{topologically exact} on the Julia set $J(f)$ (i.e., for each open set $U\in J(f)$, there exists an integer $n\geq0$ such that $J(f)\subset f^{n}(U)$).
\end{itemize}
Throughout the rest of this note, for each $f\in\sA$, we restrict
the action of $f$ to its Julia set $f|_{J(f)}:J(f)\to J(f)$. In
particular, the topological pressure of $f$ is defined through
measures supported on $J(f)$.

On the other hand, given a map $f\in\sA$, let
\begin{equation}\label{equ_uppersemicontinuous}
    \begin{split}
\cU:=&\Big\{u:J(f)\to\R\cup\{-\infty\},~~u(x)=g(x)+\sum_{c\in\tiny{\mbox{Crit}'}(f)}b(c)\log|x-c|,\\
&~\mbox{with}~g~\mbox{H\"{o}lder continuous, and~} b(c)\geq0\Big\}.
\end{split}
\end{equation}
Obviously, H\"{o}lder and geometric potentials belong to set $\cU$,
and for each upper semi-continuous potential $G\in\cU$, denote by
\begin{equation}\label{equ_singular set}
\Lambda(G):=\{x\in
J(f),~~G(x)=-\infty\}=\{c\in\mbox{Crit}'(f),~~b(c)>0\}\subseteq\mbox{Crit}'(f).
\end{equation}

The following is the key hypothesis.
\begin{defi}
Let $f: J(f)\to J(f)$ in $\sA$, then
\begin{description}
  \item[Hyperbolicity] a potential $G$ in $\cU$ is \emph{hyperbolic} for
$f$ if for some integer $n\geq1$, the function
$S_{n}(G):=\sum_{j=0}^{n-1}G\circ f^{j}$ satisfies
\begin{equation}\label{equ_hyperbolic}
\sup_{J(f)}\frac{1}{n}S_{n}(G)<P(f,G).
\end{equation}
\item[Exceptionality] a potential $G$ in $\cU$ is
\emph{exceptional} for $f$ if there is a non-empty forward invariant
finite subset $\Sigma\subset J(f)$ satisfies
\begin{equation}\label{equ_nonexceptionality}
f^{-1}(\Sigma)\backslash\Sigma\subseteq\Lambda(G).
\end{equation}
Such set $\Sigma$ is a \emph{$\Lambda(G)$-exceptional set}.
Otherwise, a potential $G$ is \emph{non-exceptional} for $f$.
\end{description}
\end{defi}
The aim of this note is aiming to prove the following:
\begin{prop}\label{prop_pressureequal}
Let $f:J(f)\to J(f)$ be an interval map in $\sA$, and $G:J(f)\to
\R\cup\{-\infty\}$ be a upper semi-continuous potential in $\cU$. If
$G$ is hyperbolic and non-exceptional for $f$, then for every
periodic point $x\in J(f)$, or every non-periodic point
$x\in J(f)\backslash\bigcup_{i=-\infty}^{+\infty}f^{i}(\Lambda(G))$,
we have
\begin{equation}\label{equ_pressure}
    P(f,G)=\lim_{n\to+\infty}\frac{1}{n}\log\sum_{y\in
    f^{-n}(x)}\exp(S_{n}(G)(y))\footnote{A term of \emph{tree pressure} is used in \cite{PU11} to stand for the right hand side of Equation \eqref{equ_pressure}.}.
\end{equation}
\end{prop}
When restricting $G$ to be a H\"{o}lder continuous potential, then $\Lambda(G)=\emptyset$ and
the non-exceptionality hypothesis is automatically satisfied. Thus, Proposition \ref{prop_pressureequal} recovers Corollary 2.2 in \cite{LRL14} for maps with non-flat critical points. Since
we want to state this proposition in the paper \cite{Zh15} to show
the existence of a conformal measure for non-exceptional geometric
potential (i.e., $G:=-t\log|Df|$)\footnote{This partially answers a
question about the existence of a conformal measure for geometric
potential at negative spectrum imposed in \cite[\S2.3]{GPR14}.} at
negative spectrum
(i.e., $t\leq0$), we decide to write down the proof in detail, even though the main methods have appeared in \cite{LRL14} already. 

\medskip

We will estimate the tree pressure in \eqref{equ_pressure} from
below and above. However, it might be worth to remark that the
estimation from above is much easier than from below. In particular,
\begin{lemm}\cite{LRL14}\label{lem_above}
Let $f:J(f)\to J(f)$ be an interval map in $\sA$, and
$G:J(f)\to\R\cup\{-\infty\}$ be a upper semi-continuous potential in
$\cU$, then for every point $x\in J(f)$, we have
$$
P(f,G)\geq\limsup\limits_{n\to\infty}\frac{1}{n}\sum_{y\in
f^{-n}(x)}\exp(S_{n}(G)(y)).
$$
\end{lemm}

The proof was written for H\"{o}lder continuous functions, but it
apply without change to the potentials in $\cU$ by using a
variational principle for upper semi-continuous functions \cite[Theo
4.4.11]{Ke98}. In the view of Lemma \ref{lem_above}, to prove
Proposition \ref{prop_pressureequal}, it is enough to show the
following.

\begin{prop}\label{prop_pressure equilvalentlow}
Let $f:J(f)\to J(f)$ be an interval map in $\sA$, and $G$ be the
upper semi-continuous potential in $\cU$, with $\Lambda(G)$ the
resulting singular set. If $G$ is hyperbolic and non-exceptional for
$f$, then for every periodic point $x\in J(f)$, or non-periodic
point
$x\in J(f)\backslash\bigcup_{i=-\infty}^{+\infty}f^{i}(\Lambda(G))$,
we have
\begin{equation}\label{equ_pressure equlivalentlow}
    \liminf\limits_{n\to\infty}\frac{1}{n}\log\sum_{y\in
    f^{-n}(x)}\exp(S_{n}(y))\geq P(f,G).
\end{equation}
\end{prop}

The proof of Proposition \ref{prop_pressure equilvalentlow} will
occupy the rest of the note, and requires a few other lemmas.

\vspace{6mm}

Given an interval map $f:J(f)\to J(f)$ in $\sA$, and a subset
$\Lambda\subseteq\mbox{Crit}'(f)$, a point $x\in J(f)$ is said to be
\emph{$\Lambda$-normal}, if for any integer $n\geq1,$ there is a
pre-image $y$ of $x$ by $f^{n}$,
such that 
$$
\{y,f(y),\cdots,f^{n-1}(y)\}\cap\Lambda=\emptyset.
$$

\medskip

\begin{lemm}\label{lemm_normal}\cite{Zh15}
Let $f:J(f)\to J(f)$ be an interval map in $\sA$, and $G:J(f)\to
\R\cup\{-\infty\}$ be a upper semi-continuous potential in $\cU$
with $\Lambda(G)$ as the resulting singular set. If $G$ is
non-exceptional for $f$, then for each $x\in J(f)$, there is an
integer $N\geq0$ such that $f^{N}(x)\notin\Lambda(G)$, and
$f^{N}(x)$ is $\Lambda(G)$-normal. In addition, if further assume
$x$ is periodic, then the integer $N=0.$
\end{lemm}
Lemma \ref{lemm_normal} permits us to deduce other lemmas.

\begin{lemm}\label{lem_second new potential}
Let $f:J(f)\to J(f)$ be an interval map in $\sA$. Let
$G:J(f)\to\R\cup\{-\infty\}$ be the upper semi-continuous potential
in $\cU$ with $\Lambda(G)$ as the resulting singular set.

Then for every integer $N\geq1$, and any compact set $K\subset
J(f)\backslash\bigcup_{j=0}^{N-1}f^{-j}(\Lambda(G)),$ there is a
constant $C_{K}>1$, such that the potential
$\widetilde{G}:=\frac{1}{N}S_{N}(G)$ satisfies the following:
\begin{enumerate}
  \item[(1)] $\widetilde{G}$ is upper semi-continuous and has log pole
  solely inside the set $\bigcup_{j=0}^{N-1}f^{-i}(\Lambda(G))$.
  Moreover $P(f,\widetilde{G})=P(f,G)$, and $G,\widetilde{G}$ share
  the same equilibrium states;
  \item[(2)] For every integer $n\geq1$, we have
  \begin{equation}\label{equ_supbound}
  \sup_{K}|S_{n}(G)-S_{n}(\widetilde{G})|<C_{K}.
  \end{equation}
\end{enumerate}
\end{lemm}
\begin{proof}
For each integer $N\geq1$, put
$$
h:=-\frac{1}{N}\sum_{j=0}^{N-1}(N-1-j)G\circ f^{j},
$$

then $\widetilde{G}=G+h-h\circ f.$

{\bf 1.} Since $G\in\cU$, and $f$ is Lipschitz and all its critical points are non-flat, we have
$\widetilde{G}$ is upper semi-continuous and has log poles solely
inside the set $\bigcup_{j=0}^{N-1}f^{-i}(\Lambda(G))$. For each
measure $\nu\in\cM(f,J(f))$, choose $\nu'$ by one of its ergodic
component if necessary. There are two cases.
\begin{itemize}
  \item Suppose $\nu'$ is atomic, then the topological
  exactness on $J(f)$ yields that $\nu'$ supports on a periodic orbit
  $O_{x}$ in $J(f)$. Since no critical point of $f$ in $J(f)$ is periodic, the
  set
  $O_{x}\bigcap\bigcup_{j=0}^{\infty}f^{-j}(\Lambda(G))=\emptyset$.
  This implies functions $\widetilde{G},~G$ and $h$ have no log poles on $O_{x}$.
  Thus
  $$
  \int_{J(f)}\widetilde{G}d\nu'=\int_{J(f)}G+h-h\circ
  fd\nu'=\int_{J(f)}Gd\nu'.
  $$
  \item Suppose $\nu'$ is non-atomic, then the topological exactness
  on $J(f)$
  yields that $\nu'$ supports on entire $J(f)$, and
  $\nu'(\bigcup_{j=0}^{\infty}f^{-j}(\Lambda(G)))=0.$ This also
  implies that
   $$
\int_{J(f)}\widetilde{G}d\nu'=\int_{J(f)}G+h-h\circ
fd\nu'=\int_{J(f)}Gd\nu'.
  $$
  So in both cases, we have $P(f,G)=P(f,\widetilde{G})$, and thus
  $G,\widetilde{G}$ share the same equilibrium states.
\end{itemize}
{\bf 2.} Since for each $n\geq N$,
$$
S_{n}(\widetilde{G})=S_{n}(G+h\circ f-h)=S_{n}(G)+h\circ f^{n}-h,
$$
and $h\circ f^{n},h$ are H\"{o}lder continuous on the compact set
$K$, so we have the desired inequality \eqref{equ_supbound} with
$C_{K}:=(N-1)(\sup_{K}G-\inf_{K}G).$
\end{proof}

\bigskip

From now on, given an interval map $f:J(f)\to J(f)$ in $\sA$ and a
upper semi-continuous potential $G:J(f)\to\R\cup\{-\infty\}$, denote
by $\cL_{G}$ the \emph{transfer operator} acting on the space of
bounded functions on $J(f)$ and taking values in $\mathbb{C}$,
defined as
$$
\cL_{G}(\psi)(x):=\sum_{y\in f^{-1}(x)}\exp(G(y))\psi(y).
$$

\medskip

We are ready to prove Proposition \ref{prop_pressure
equilvalentlow}. In informal term, the proof is split into 3 parts.
In Part 1, we construct a new potential $\widetilde{G}$ by the
Birkoff average of $G$, and show that $\widetilde{G}$ is also
hyperbolic and non-exceptional for $f$. In part 2, we rely on the
hyperbolicity to ensure the existence of an ergodic measure with
positive Lyapunov exponent. Applying the Pesin Theory and Katok
Theory on this measure, we will obtain a low bound estimation on the
tree pressure by the differeomphic pull-backs on a neighborhood of a
$\mbox{Crit}'(f)$-normal point. In Part 3, we will use the
non-exceptionality and topological exactness to move the desired
points inside a neighbor of a $\mbox{Crit}'(f)$-normal point, and
away from the singular set.

\medskip

\begin{proof}[Proof of Proposition \ref{prop_pressure equilvalentlow}.]
{\bf 1.} Since $G$ is hyperbolic for $f$, there exists an integer
$N\geq1$, so that the function $\widetilde{G}:=\frac{1}{N}S_{N}(G)$
satisfies that $\sup_{J(f)}\widetilde{G}<P(f,G)$. Applying Part
$(1)$ of Lemma \ref{lem_second new potential}, the function
$\widetilde{G}$ is upper semi-continuous and has log poles solely
inside the set $\bigcup_{j=0}^{N-1}f^{-j}(\Lambda(G))$, and
\begin{equation}\label{equ_newpotentialhyperbolic}
P(f,\widetilde{G})=P(f,G)>\sup_{J(f)}\widetilde{G}.
\end{equation}
Next, we show that the potential $\widetilde{G}$ is also
non-exceptional for $f$. This can be proved by contradiction.
Suppose on the contrast, then there exists a non-empty finite
forward invariant subset $\Sigma\subset J(f)$, such that
\begin{equation}\label{equ_exceptonalcondition}
    f^{-1}(\Sigma)\backslash\Sigma\subset
    \Lambda(\widetilde{G})=\bigcup_{j=0}^{N-1}f^{-j}(\Lambda(G)).
\end{equation}
Without loss of generality, we can assume that $\Sigma$ contains
only one periodic point $p$, and every point in $\Sigma$ is
pre-periodic and will map to $p$. Using
\eqref{equ_exceptonalcondition}, it follows that
$$
\forall x\in f^{-1}(\Sigma)\backslash\Sigma,~~\exists
j:=j_{x}\geq0,~~s.t.~~f^{j}(x)\in\Lambda(G)\subset\mbox{Crit}'(f).
$$
Note also that no critical point in $J(f)$ is periodic, so for each
$x\in f^{-1}(\Sigma)\backslash\Sigma$, we can define by $j^*$ the
unique index $j$ such that $f^{j}(x)\in\Lambda(G)$, but
$f^{i}(x)\notin\Lambda(G),~~\forall i>j$.
Put also
$$A:=\{f^{j^{*}}(x):~~x\in
f^{-1}(\Sigma)\backslash\Sigma\},~~\mbox{and}~~
\Sigma':=\bigcup_{i=1}^{\infty}f^{i}(A).$$

With this convention, to complete the proof, it is sufficient to
show that $\Sigma'$ is a $\Lambda(G)$-exceptional set, so that it is
a contradiction to the hypothesis that the potential $G$ is
non-exceptional for $f$. This contradiction yields that the
potential $\widetilde{G}$ is non-exceptional for $f$.

\medskip

It is straightforward to see that the set $\Sigma'$ is non-empty,
finite and forward invariant, so it only remains to verify that
$f^{-1}(\Sigma')\backslash\Sigma'\subseteq\Lambda(G)$. Note that
$\Sigma'\subseteq\Sigma$, so it follows from
\eqref{equ_exceptonalcondition}, that
\begin{equation}\label{equ_twocases}
    \forall y\in \Sigma',~~\forall z\in f^{-1}(y)\Rightarrow
z\in\Sigma\backslash\Lambda(\widetilde{G}),~~\mbox{or}~~z\in\Lambda(\widetilde{G}).
\end{equation}
\begin{itemize}
  \item Suppose $z\in\Sigma\backslash\Lambda(\widetilde{G})$. Note
  that by definition
  $\Sigma'\cap\Lambda(G)=\emptyset$, so the forward orbit of $y$, and $z$ are
  outside the set $\Lambda(G)$. This means there must exists $z'\in
  f^{-1}(\Sigma)\backslash\Sigma$ and a integer $d\geq1$ such that
  $z=f^{j_{z'}^{*}+d}(z')$. In other words $z\in\Sigma'$;
  \item Suppose $z\in\Lambda(\widetilde{G}).$ This implies that
  there is an integer $j\geq0$, with
  $$
f^{j}(z)=f^{j-1}(f(z))=f^{j-1}(y)\in\Lambda(G);
  $$
  On the other hand, since $y\in\Sigma'$, there is $x\in
  f^{-1}(\Sigma)\backslash\Sigma$ and integer $d\geq1$ such that
  $y=f^{j^{*}+d}(x)$. Therefore
  $$
f^{j-1+j^{*}+d}(x)\in\Lambda(G).
  $$
By the maximality of $j^{*}$, we have $j-1+d\leq0$, so $d=1$ and
$j=0$. In other words, $z\in\Lambda(G)$.
\end{itemize}
In conclusion, the right hand side of \eqref{equ_twocases} yields
that $z\in\Sigma'$ or $z\in\Lambda(G)$, namely $\Sigma'$ is a
$\Lambda(G)$-exceptional set, as we wanted.

\medskip

{\bf 2.} We are aiming to prove following Claim in this section.

{\bf Claim:} For every $\epsilon>0$, and every
$\mbox{Crit}'(f)$-normal point $x$ of $J(f)$, there is $\delta>0$
such that
\begin{equation}\label{equ_tree pressure on diff part}
    \liminf\limits_{n\to\infty}\frac{1}{n}\sum_{W\in\sD_{n}}\inf\limits_{W\cap
    J(f)}\exp(S_{n}(\widetilde{G}))\geq P(f,\widetilde{G})-\epsilon,
\end{equation}
where $\sD_{n}$ is the collection of diffeomorphic pull-backs of
$B(x,\delta)$ by $f^{n}$.

On one hand, Inequality \eqref{equ_newpotentialhyperbolic} yields
that there is $\epsilon>0$ so that
$\epsilon<P(f,\widetilde{G})-\sup_{J(f)}\widetilde{G}$. Let $\nu$ be
a measure in $\cM(J(f),f)$ such that
$$
h_{\nu}(f)+\int_{J(f)}\widetilde{G}d\nu\geq
P(f,\widetilde{G})-\epsilon>\sup_{J(f)}\widetilde{G}.
$$
Replacing $\nu$ by one of its erogodic components if necessarily,
assuming that $\nu$ is ergodic. We thus have
$$
h_{\nu}(f)>\sup_{J(f)}\widetilde{G}-\int_{J(f)}\widetilde{G}d\nu\geq0,
$$
and then the Ruelle's inequality yields that the Lyapunov exponent
of $\nu$ is strictly positive. Applying
\cite[Theo11.6.1]{PU11}\footnote{Actually, the proof is written for
complex rational maps with geometric potential, but they apply
without changes to interval function $\widetilde{G}$ by applying
\cite[Theo 6]{Do08} instead of \cite[Coro11.2.4]{PU11}.}, there is a
compact and forward invariant subset $Y$ of $J(f)$ on which $f$ is
topological transitive, so $f$ is open and uniformly expanding, and
so that
$$
P(f|_{Y},\widetilde{G}|_{Y})\geq P(f,\widetilde{G})-\epsilon.
$$
Therefore, \cite[Theo4.4.3]{PU11} implies that there is
$\delta_{0}>0$ such that the desired property \eqref{equ_tree
pressure on diff part} holds for every $x\in Y$ with
$\delta=\delta_{0}$.

One the other hand, the hypothesis that $x$ is
$\mbox{Crit}'(f)$-normal and $f$ is topological exact on $J(f)$
imply that there is a non-critical pre-image $x'\in B(Y,\delta_{0})$
of $x$ such that all the pre-images are non-critical and
$\{x',f(x'),\cdots,x\}\cap\mbox{Crit}'(f)=\emptyset$. Therefore,
$\widetilde{G}$ is finite along the orbit $\{x',f(x'),\cdots,x\}$
and there exists $\delta>0$ such that the pull-back of
$B(x_{0},\delta)$ by $f^{n}$ that contains $x'$ is contained in
$B(x,\delta_{0})$, and the desired assertion \eqref{equ_tree
pressure on diff part} directly follows from the previous
discussion. So, the proof of the claim is completed.

\medskip

{\bf 3.} Let $x$ be a periodic point or a non-periodic point
$J(f)\backslash\bigcup_{j=-\infty}^{\infty}f^{j}(\Lambda(G))$, and
recall $N$ to be the integer given in Part 1. Since no critical
point of $f$ in $J(f)$ is periodic, following the discussions in
Part 1, there are a compact subset $\hat{K}\subset~J(f)\backslash
\bigcup_{j=0}^{N-1}f^{-j}(\Lambda(G))$ and a constant $C_{\hat{K}}$,
such that $x\in\hK$, and
$$
|S_{n}(G)(x)-S_{n}(\widetilde{G})(x)|\leq\sup_{\hat{K}}|S_{n}(G)-S_{n}(\widetilde{G})|<C_{\hat{K}},~~\forall
n\geq1.
$$
With this convention, to complete the proof of the lemma, it
suffices to prove that for every $\epsilon>0$, there is $N_{0}>0$
such that for every $n\geq N_{0}$, we have
\begin{equation}\label{equ_equlivalent}
\cL_{\widetilde{G}}^{n}({\bf
1})(x)\geq\exp(C_{\hK})\exp(n(P(f,\widetilde{G})-\epsilon)).
\end{equation}
On one hand, let $x_{0}$ be a $\mbox{Crit}'(f)$-normal point. Let
$\delta>0$ and for each $n\geq1,$ let $\sD_{n}$ be as in
\eqref{equ_tree pressure on diff part} with $\epsilon$ replacing by
$\epsilon/2$. Then there is $n_{0}\geq1$ such that for every integer
$n\geq n_{0}$, we have
$$
\frac{1}{n}\log\sum_{W\in\sD_{n}}\inf_{W\cap J(f)}\exp(S_{n})(\widetilde{G})\geq
P(f,\widetilde{G})-\epsilon/2.
$$
This implies that for each $n\geq n_{0}$, and every $x^*\in
B(x_{0},\delta)\cap~J(f)$, we have
$$
\cL_{\widetilde{G}}^{n}({\bf
1})(x^{*})\geq\exp(n(P(f,\widetilde{G})-\epsilon/2)).
$$
On the other hand, note that it follows from Part 1 that the
potential $\widetilde{G}$ is non-exceptional for $f$. Applying Lemma
\ref{lemm_normal}, point $x$ must be
$\Lambda(\widetilde{G})$-normal. Together with the topological
exactness on $J(f)$, there exists $n_{1}\geq 1$, and $x'\in
B(x_{0},\delta)\cap J(f)$ with
\begin{equation}\label{equ_avoiding the critical}
f^{n_{1}}(x')=x~~\mbox{and}~~f^{i}(x')\notin\Lambda(\widetilde{G}),~~\forall
i=0,1,2,\cdots,n_{1}.
\end{equation}
Therefore, for every $n\geq n_{1}+n_{0}$,
\begin{align*}
\cL_{\widetilde{G}}^{n}({\bf 1})(x)&=\sum_{y\in
f^{-n}(x)}\exp(S_{n}(\widetilde{G})(y))\\
&=\sum_{y'\in f^{-n_{1}}(x)}\sum_{y\in
f^{-(n-n_{1})}(y')}\exp(S_{n-n_{1}}(\widetilde{G})(y)+S_{n_{1}}(\widetilde{G})(y'))\\
&\geq
\underbrace{\exp(S_{n_{1}}(\widetilde{G})(x'))\cL_{\widetilde{G}}^{n-n_{1}}({\bf
1})(x')}_{\circled{1}_{n-n_{1}}}.
\end{align*}
Using \eqref{equ_avoiding the critical}, we can choose another
compact set $K\subset J(f)$ contains $\{f^{i}(x')\}_{i=0}^{n_{1}}$
but away from $\Lambda(\widetilde{G})$, so that
$\inf_{K}(\widetilde{G})>-\infty$. Thus
\begin{align*}
\circled{1}_{n-n_{1}}&\geq\exp(n_{1}\inf_{K}\widetilde{G})\cL_{\widetilde{G}}^{n-n_{1}}({\bf
1})(x')\\
&\geq\exp(n_{1}\inf_{K}\widetilde{G})\exp((n-n_{1})(P(f,\widetilde{G})-\epsilon/2)).
\end{align*}
Let $N_2>0$ be such that
$$
\exp(n_{1}\inf_{K}\widetilde{G})\geq\exp(C_{\hK})\exp(n_{1}P(f,\widetilde{G})-(\epsilon
N_{2})/2).
$$
Therefore for every integer $n\geq\max\{n_{1}+n_{0},N_{2}\}$, we
have
\begin{align*}
\cL_{\widetilde{G}}^{n}({\bf 1})(x)&
\geq\exp(C_{\hK})\exp(n_{1}P(f,\widetilde{G})-(\epsilon
N_{2})/2+(n-n_{1})(P(f,\widetilde{G})-\epsilon/2))\\
&\geq
\exp(C_{\hK})\exp(n(P(f,\widetilde{G})-\epsilon)+(\epsilon/2)(n+n_{1}-N_{2}))\\
&\geq \exp(C_{\hK})\exp(n(P(f,\widetilde{G})-\epsilon)).
\end{align*}
This provides the desired inequality \eqref{equ_equlivalent} with
$N_{0}=\max\{n_{1}+n_{0},N_{2}\}$, and the proof of this lemma is
completed.
\end{proof}

\bigskip

\bigskip

\bigskip

\begin{small}
    \noindent
    \setlength{\tabcolsep}{0cm}
    \begin{tabular}{l@{\hspace{2cm}}l}
        \textsc{Yiwei Zhang} \\
        \email{\href{mailto:yzhang@mat.puc.cl}{yzhang@mat.puc.cl}} \\
    \end{tabular}
    \bigskip

    \noindent
    \textsc{Facultad de Matem\'aticas, Pontificia Universidad Cat\'olica de Chile}

    \noindent
    \textsc{Avenida Vicu\~na Mackenna 4860, Santiago, Chile}
\end{small}

\end{document}